\documentclass[runningheads]{llncs}
\usepackage{amssymb}
\usepackage{amsmath}
\usepackage{mathtools}
\usepackage{bm}
\DeclareMathAlphabet\mathbfcal{OMS}{cmsy}{b}{n}
\usepackage{algorithm}
\usepackage{algpseudocode}
\usepackage[caption=false]{subfig}
\newfloat{algorithm}{tb}{alg}
\usepackage{graphicx}
\usepackage{hyperref}
\usepackage{pgfplots}
\pgfplotsset{compat=newest}
\pgfplotsset{
	shortlegend/.style={%
		legend image code/.code={
			\draw[##1,line width=0.8pt]
				plot coordinates {
					(0cm,0cm)
					(0.3cm,0cm)
					};%
				}
		}
}
\usepackage{tikz}
\usepackage{tikz-qtree}
\usetikzlibrary{arrows,automata,matrix,positioning,patterns,shapes.multipart}
\usepackage{tabularx}
\usepackage{booktabs}
\newcolumntype{j}{>{\centering\arraybackslash}p}
\newcolumntype{z}{>{\raggedleft\arraybackslash}p}
\newcolumntype{J}{>{\centering\arraybackslash}X}
\usepackage{multirow}
\usepackage{xcolor}

\definecolor{Color0}{HTML}{E41A1C}
\definecolor{Color1}{HTML}{377EB8}
\definecolor{Color2}{HTML}{4DAF4A}
\definecolor{Color3}{HTML}{984EA3}
\definecolor{Color4}{HTML}{FF7F00}

\newcommand{\vek}[1]{{\ensuremath{\bm{#1}}}}
\newcommand{\iv}[1]{\ensuremath{\left[#1\right]}}

\newcommand{\lb}[1]{\ensuremath{\underline{#1}}}
\newcommand{\ub}[1]{\ensuremath{\overline{#1}}}
\newcommand*{\QED}{\null\nobreak\hfill\ensuremath{\square}}
\usepackage{comment}
\algrenewcommand\algorithmicindent{1em}
\newcommand{\dcocpp}{{\ttfamily
dco\kern-.08em{\raisebox{-.1ex}{/}\kern-.15em
{c\kern-.03em{\raisebox{-.18ex}{+\kern-.028em{+}}}}}}}
\begin{document}
\title{Subdomain Separability in Global Optimization}
%
\author{Jens Deussen \and Uwe Naumann}
\authorrunning{J. Deussen et al.}
%
\institute{Informatik 12: Software and Tools for Computational Engineering,\\
RWTH Aachen University, Aachen, Germany\\
\email{\{deussen,naumann\}@stce.rwth-aachen.de}}
\maketitle              
\begin{abstract}
We propose a generalization of separability in the context of global
optimization.
Our results apply to objective functions implemented as differentiable computer
programs.
They are presented in the context of a simple branch and bound method.
The often significant search space reduction can be expected to yield an
acceleration of any global optimization method.
We show how to utilize interval derivatives calculated by adjoint algorithmic
differentiation to examine the monotonicity of the objective with respect to
so called structural separators and how to verify the latter automatically.

\keywords{
\and global optimization
\and algorithmic differentiation
\and branch and bound
\and interval adjoints
\and search space reduction
\and separable functions.}
\end{abstract}

\section{Introduction}
\label{sec:intro}
In contrast to local optimization methods, deterministic global optimization
methods, e.g.\ interval-based branch and bound (b\&b) algorithms
\cite{Falk1969}, guarantee to find the global solution for a predefined
tolerance for optimality in finite time \cite{Floudas2000}.
These methods are more expensive in terms of computational effort than their
local counterparts.

An important property that should be exploited during optimization is
separability of the objective function.
A function $f: \mathbb{R}^n \to \mathbb{R}$ is called partially separable
(also: decomposable) if it is of the form
\begin{align}
f(\vek{x}) = f_1\big((x_i)_{i\in X_1}\big) + f_2\big((x_i)_{i\in X_2}\big)\ ,
\label{eq:sep}
\end{align}
with a given partitioning of the set of indexes of independents into two
disjoint subsets $X_1$, $X_2$ and functions
$f_1: \mathbb{R}^{|X_1|} \to \mathbb{R}$ and
$f_2: \mathbb{R}^{|X_2|} \to \mathbb{R}$.
The function is called (fully) separable if the separation can be applied
recursively until all disjoint subsets only contain a single element
\cite{Jamil2013,Li2013}.
For a global optimization problem
\begin{align*}
y^\ast = \min_{\vek{x} \in D\subseteq\mathbb{R}^n} f(\vek{x})\ ,
\end{align*}
with partially separable objective function $f$ as in \eqref{eq:sep} it is
well known \cite{Hadley1964} that the global minimum can be obtained
by decomposing the problem into smaller subproblems
\begin{align*}
y^\ast
=\min_{(x_i \in D_i)_{i\in X_1}\subseteq\mathbb{R}^{|X_1|}}
f_1\big((x_i)_{i\in X_1}\big)
+\min_{(x_i \in D_i)_{i\in X_2}\subseteq\mathbb{R}^{|X_2|}}
f_2\big((x_i)_{i\in X_2}\big)\ ,
\end{align*}
that can be solved in parallel.
In the context of b\&b algorithms with a division into $k$ parts
for all dimensions every non-leaf node generates $k^n$ children.
The decomposition reduces the number of generated nodes to
$\mathcal{O}(k^{\max(|X_1|,|X_2|)})$ for the particular
problem and thus results in a potentially significant reduction of the
corresponding search space.

Separable functions have been extensively researched in the context of
optimization.
In \cite{Griewank1982} a quasi-Newton method is introduced that exploits the
structure of partially separable functions when computing secant updates for
the Hessian matrix.
A parallel b\&b approach was used in \cite{Phillips1990} to find
optima of non-convex problems with partially separable functions over a bounded
polyhedral set.
In \cite{Colson2006} a derivative-free method for exploiting partial
separability in unconstrained optimization was proposed.
The automatic detection of partial separability as in \eqref{eq:sep} by
algorithmic differentiation was proposed in \cite{Gay1996}.

In \cite{Salomon1996} a class of problems was introduced, which is called
\emph{as easy to optimize as decomposable functions}
and that is related to the present work.
Such functions satisfy
\begin{align}
\frac{df}{dx_i}(\vek{x}) = g(x_i)\cdot h(\vek{x})\ ,
\label{eq:salomon_sep}
\end{align}
such that the first-order optimality condition
\begin{align*}
\frac{df}{dx_i}(\vek{x}) = 0\ ,
\end{align*}
can be transformed to $g(x_i)=0$.
The equation is only dependent on a single variable.
Optima for which $h(\vek{x})=0$ and optima at the boundary are not taken into
consideration by this approach.

In this paper we aim to generalize the concept of separability in order to make
previously non-separable functions also benefit from decomposition of the
optimization problem on subdomains.
Therefore, the function must be of a special structure which is less
restrictive than \eqref{eq:sep}, but is a variation of
\eqref{eq:salomon_sep} and additionally needs to fulfill a monotonicity
condition on the separator.
The monotonicity condition guarantees that the decomposition still takes all
possible optima into consideration which is crucial for the integration into
deterministic global optimization algorithms.

We use interval adjoints as a combination of reliable
interval computations \cite{Moore1979,Moore2009} and adjoint algorithmic
differentiation \cite{Griewank2008,Naumann2012} to obtain an enclosure of
all adjoints over a given subdomain.
In \cite{Vassiliadis2016} we used this information for significance based
approximate computing.
In \cite{Afghan2020} we discussed significance analysis in the context of
neural networks.
Deterministic global optimization through a check for first-order
optimality is described in \cite{Deussen2019}.
In the following we show how to use interval adjoints for a monotonicity
check of structural separators and for the verification of these separators.

The paper is organized as follows:
In Section~\ref{sec:contribution} we define structural separability and we
formulate the necessary monotonicity condition for the decomposition of the
optimization problem.
Examples for functions that are non-separable by \eqref{eq:sep} but fulfill
the new definition such that their corresponding optimization problem can still
be decomposed are given.
Section~\ref{sec:implementation} explains how to implement the presented work
and how to integrate it into a b\&b algorithm for deterministic global
optimization.
Therefore, interval adjoints are utilized for the examination of the
monotonicity condition and for automatic detection of separators.
In Section~\ref{sec:casestudy} we show results from a proof of concept
implementation for the examples from Section~\ref{sec:contribution}
followed by conclusion and outlook in Section~\ref{sec:conclusion}.

\section{Subdomain Separability}
\label{sec:contribution}
We introduce subdomain separability and we show how to exploit this property in
global optimization.
\begin{definition} \label{def:gs}
A function $f: \mathbb{R}^n \to \mathbb{R}$ is called structurally separable
if it is of the form
\begin{align*}
f(\vek{x}) = f_s\left(s\big((x_i)_{i\in X_1}\big),(x_i)_{i\in X_2}\right)
\end{align*}
with disjoint and non-empty index sets $X_1$ and $X_2$.
The scalar function $s\big((x_i)_{i\in X_1}\big)$ is called structural
separator.
\end{definition}

Conventionally separable functions as in \eqref{eq:sep} are covered by
Definition~\ref{def:gs} with structural separators
$s_1=f_1\big((x_i)_{i\in X_1}\big)$, $s_2=f_2\big((x_i)_{i\in X_2}\big)$ and
\begin{align*}
f_s\big(s_1,(x_i)_{i\in X_2}\big) = s_1 + f_2\big((x_i)_{i\in X_2}\big)\ ,\\
f_s\big(s_2,(x_i)_{i\in X_1}\big) = s_2 + f_1\big((x_i)_{i\in X_1}\big)\ .
\end{align*}

Application of the chain rule of differentiation to differentiable structurally
separable functions yields the gradient
\begin{align*}
\frac{df}{d\vek{x}}(\vek{x}) = \begin{pmatrix}
\frac{df}{ds}(\vek{x})\cdot
\frac{ds}{dx_i}\big((x_j)_{j\in X_1}\big) &\forall i \in X_1\\
\frac{df}{dx_i}(\vek{x}) &\forall i \in X_2
\end{pmatrix}\ .
\end{align*}
If $X_1$ only contains a single element, then the structurally separable
function $f$ also satisfies \eqref{eq:salomon_sep} with
$g(x_i)=\frac{ds}{dx_i}(x_i)$ and
$h(\vek{x})=\frac{df}{ds}(\vek{x})$.

\begin{theorem}
\label{th:sep}
Consider the global optimization problem
\begin{align}
\min_{\vek{x} \in D \subseteq \mathbb{R}^n} f(\vek{x})
 = f_s\left(s\big((x_i)_{i\in X_1}\big),(x_i)_{i\in X_2}\right)\ ,
\label{eq:opt}
\end{align}
with structurally separable, non-convex and differentiable objective function
$f$ and separator $s\big((x_i)_{i\in X_1}\big)$.
If the objective function is monotonic w.r.t.\ the separator on the domain,
that is,
\begin{align}
\frac{df}{ds}(\vek{x})\geq0\quad\forall\vek{x}\in D
\ \vee\ \frac{df}{ds}(\vek{x})\leq0\quad\forall\vek{x}\in D\ ,\label{eq:cond}
\end{align}
and
\begin{align*}
\frac{df}{ds}(\vek{x})=\frac{\partial f_s}{\partial s}(s,(x_i)_{i\in X_2})\ ,
\end{align*}
then the optimization problem in \eqref{eq:opt} can be decomposed into
\begin{align}
\min_{(x_i \in D_i)_{i \in X_2}} &
f_s\left(s^\ast,(x_i)_{i\in X_2}\right)\ ,
\label{eq:optgps1}\\
\mathrm{s.t.}\ & s^\ast =
\begin{cases}
\displaystyle\min_{(x_i \in D_i)_{i \in X_1}}
s\big((x_i)_{i\in X_1}\big)\ 
\mathrm{if}\ \frac{df}{ds}(\vek{x})\geq0\ \forall \vek{x} \in D\ ,\\
\displaystyle\max_{(x_i \in D_i)_{i \in X_1}}
s\big((x_i)_{i\in X_1}\big)\ 
\mathrm{if}\ \frac{df}{ds}(\vek{x})\leq0\ \forall \vek{x} \in D\ .
\end{cases}\label{eq:optgps2}
\end{align}
\end{theorem}

\begin{proof}
From \eqref{eq:cond} we know that the objective function is either
monotonically increasing or decreasing w.r.t.\ the separator.
In case it is monotonically increasing, that is $\frac{df}{ds}(\vek{x})\geq0$,
over the subdomain $D$, we have
\begin{align*}
f_s\left(s^-,(x_i)_{i\in X_2}\right)\leq f_s\left(s^+,(x_i)_{i\in X_2}\right)
\ \forall (x_i \in D_i)_{i\in X_2} \ ,
\end{align*}
for $s^- \leq s^+$.
As to $\frac{\partial f_s}{\partial x_i}(s,(x_j)_{j\in X_2})=0$ for
$i\in X_1$, and due to monotonicity
$\frac{\partial f_s}{\partial s}(s,(x_j)_{j\in X_2})>0$ the global minimum of
$f$ requires the separator $s$ to be minimal on the domain.
The monotonic decrease scenario is handled analogously.
\QED
\end{proof}
\begin{remark}
The dimension of the inner optimization problem as in \eqref{eq:optgps2} is
$|X_1|$ while the dimension of the outer optimization problem in
\eqref{eq:optgps1} is $|X_2|+1$.
\end{remark}
\begin{remark}
If $s\big((x_i)_{i\in X_1}\big)$ is also structurally separable, then the
separation approach can be applied recursively and the original optimization
problem decomposes into even smaller disjoint optimization problems.
\end{remark}
\begin{remark}
If two structural separators $s_1\big((x_i)_{i\in X_1}\big)$ and
$s_2\big((x_i)_{i\in X_2}\big)$ fulfilling \eqref{eq:cond} are
independent of each other, i.e.\ $X_1 \cap X_2 = \emptyset$, the decomposed
optimization problems can be solved in parallel.
Otherwise, either separator $s_1$ or $s_2$ needs to be optimized first if
$X_1\subset X_2$ or $X_2\subset X_1$, respectively.
\end{remark}
\begin{remark}
If the monotonicity condition in \eqref{eq:cond} holds for separator
$s=x_i$, $i \in \{0,\ldots,n-1\}$ then the minimum is located at the boundary
either at $\min_{x_i\in D_i} x_i$ for $\frac{df}{ds}(\vek{x})\geq0$ or
$\max_{x_i\in D_i} x_i$ for $\frac{df}{ds}(\vek{x})\leq0$.
\end{remark}
\begin{remark}
A degenerate solution is implied if $\frac{df}{ds}(\vek{x})=0$ for all
$\vek{x}\in D$ and $D$ contains more than one element.
\end{remark}
\begin{remark}
If the monotonicity condition is violated, then the structural separability can
still be exploited similar to \cite{Salomon1996} by solving
$\frac{ds}{dx_i}\big((x_i)_{i\in X_1}\big)=0$ for finding stationary points.
As already proposed in Section~\ref{sec:intro} this approach does not
necessarily compute all stationary points.
\end{remark}

\paragraph{Examples}
Five test problems are investigated in the light of subdomain separability.
They illustrate different aspects of the general approach.
Besides the partially separable function in Example~\ref{ex:ST}, there is the
exponential function which is solvable in parallel and globally monotonic in
Example~\ref{ex:Exp}, a recursive exponential function which is still globally
monotonic but cannot be solved in parallel in Example~\ref{ex:RExp} and the
Shubert function in Example~\ref{ex:Shu} that is not globally monotonic but
solvable in parallel.
Example~\ref{ex:Sal} can neither be solved in parallel nor is it globally
monotonic but it could still benefit from subdomain separability.

\begin{example}[Styblinski-Tang function \cite{StyblinskiTangFunction}]
\label{ex:ST}
Partially separable functions as in \eqref{eq:sep} are structurally separable
and always fulfill the monotonicity condition in \eqref{eq:cond} with
\begin{align*}
\frac{df}{ds_1}(\vek{x}) = 1\quad\wedge\quad \frac{df}{ds_2}(\vek{x}) = 1\ ,
\end{align*}
on any domain which yields the well-known fact that the corresponding
optimization problem can be decomposed and solved in parallel.

For example, the Styblinski-Tang function
\begin{align*}
f(\vek{x})=\frac12\sum_{i=0}^{n-1}\left(x_i^4 -16x_i^2+5x_i\right)\ ,
\end{align*}
is as in \eqref{eq:sep} except for the factor in front of the sum.
In \cite{Jamil2013} it is marked as non-separable.
Still, the problem can be decomposed into
$f(\vek{x})=\frac12\sum_{i=0}^{n-1}s_i$ with $\frac{df}{ds_i}=\frac12$
for any $\vek{x}\in\mathbb{R}^n$.
\end{example}

\begin{example}[Exponential function \cite{ExponentialFunction}]
\label{ex:Exp}
For the exponential function
\begin{align*}
f(\vek{x}) = -\exp\left(-\frac12\sum_{i=0}^{n}{x_i^2}\right)\ ,
\end{align*}
we choose $s_i=x_i^2$ to be the separators and the derivative of the
objective w.r.t.\ these separators is equal to
\begin{align*}
\frac{df}{ds_i}(\vek{x})
=\frac12\exp\left(-\frac12\sum_{i=0}^{n}{x_i^2}\right)\ .
\end{align*}
The exponential function is globally monotonically increasing.
Theorem~\ref{th:sep} becomes applicable to all separators.
The resulting subproblems can be solved in parallel.
\end{example}

\begin{example}[Recursive exponential function]
\label{ex:RExp}
To demonstrate the usefulness of structural separability we consider the
optimization problem in \eqref{eq:opt} with $D=\iv{-2,3}^n$,
objective function $f: \mathbb{R}^n \to \mathbb{R}$ and differentiable program
\begin{align*}
f(\vek{x}) &= y_{n}\ ,\\
y_{i+1} &= \exp(x_i^2 + y_i - 1)\ ,\\
y_0 &= 1\ ,
\end{align*}
which is non-separable in a conventional manner, but fulfills
Definition~\ref{def:gs} with separators $y_i$, $i=1,\ldots,n-1$.
To decompose the optimization problem it remains to be shown that the
derivatives of the objective with respect to the separators
$\frac{df}{dy_i}(\vek{x})$ for $i=0,\ldots,n$
are positive (or negative) on any subdomain.
From
\begin{align*}
\frac{\partial y_{i+1}}{\partial y_i} = y_{i+1}\ ,
\end{align*}
it follows that
\begin{align*}
\frac{df}{dy_i}(\vek{x})
= \prod_{j=i}^{n-1} \frac{\partial y_{j+1}}{\partial y_j}
= \prod_{j=i}^{n-1} y_{j+1}\ .
\end{align*}
By mathematical induction we show that $y_i \geq 1$ for $i=0,\ldots,n$.
The basis $y_0=1$ obviously fulfills the statement.
The assumption $y_i \geq 1$ yields
\begin{align*}
y_{i+1} = \exp(x_i^2 + y_i - 1) \geq \exp(x_i^2) \geq \exp(0) = 1\ ,
\end{align*}
due to monotonicity of the exponential function.
Thus, $y_i \geq 1$ and $\frac{df}{dy_i}(\vek{x}) \geq 1$
for $i=0,\ldots,n$.
Furthermore, we know that the global minimum is located at $\vek{x}=0$
with a value of $f(\vek{x})=1$.

As a consequence of Theorem~\ref{th:sep} the optimization problem can be
reformulated as
\begin{align*}
\min_{\vek{x}\in \iv{-2,3}^n} & f(\vek{x}) = y_n^\ast\ ,\\
\mathrm{s.t.}\ & y^\ast_{i+1} =
\min_{x_i \in \iv{-2,3}}\ \exp(x_i+y^\ast_i-1)\ ,\\
&y_{0\phantom{+1}}^\ast = 1\ .
\end{align*}

Note, that this function is globally monotonic w.r.t.\ the separator which does
not necessarily hold in general.
Since the separators are partially dependent on each other the corresponding
optimization problems need to be solved sequentially beginning with $y_1^\ast$.
\end{example}


\begin{example}[Shubert function \cite{ShubertFunction}]
\label{ex:Shu}
The Shubert function is given by
\begin{align*}
f(\vek{x})=\prod_{i=0}^{n-1}\sum_{j=1}^5{\cos((j+1)x_i+j)}\ .
\end{align*}
Each factor of the multiplication can be considered as a structural separator
with $s_i=\sum_{j=1}^5{\cos((j+1)x_i+j)}$.
Derivatives of the function value w.r.t.\ the separators are derived as
\begin{align*}
\frac{df}{ds_i}(\vek{x}) = \prod_{\substack{j=0 \\ j\neq i}}^{n-1} s_j\ .
\end{align*}
If any $s_i$ is either positive or negative, then the corresponding
optimization problem can be decomposed by Theorem~\ref{th:sep}.
\end{example}

\begin{example}[Salomon function \cite{Salomon1996}]
\label{ex:Sal}
We show that the Salomon function is separable only on selected subdomains.
The differentiable program is given by
\begin{align*}
f(\vek{x}) = 1-\cos\left(2\pi\sqrt{\sum_{i=0}^{n-1} x_i^2}\right)
+0.1\sqrt{\sum_{i=0}^{n-1} x_i^2}\ .
\end{align*}
Introduction of an intermediate result $S=\sqrt{\sum_{i=0}^{n-1} x_i^2}$ and
of separators $s_i = x_i^2$ yields the derivatives
\begin{align*}
\frac{df}{ds_i}(\vek{x})&=\frac{df}{dS}(\vek{x})\cdot\frac{dS}{ds_i}(\vek{x})\ ,\\
\frac{df}{dS}(\vek{x})&=2\pi\sin(2\pi S)+0.1\ ,\\
\frac{dS}{ds_i}(\vek{x})&=\frac{1}{2S}\ .
\end{align*}
As $\frac{dS}{ds_i}(\vek{x})$ is always positive it remains to be shown that
$\frac{df}{dS}(\vek{x})$ is either positive or negative.
The roots of $\frac{df}{dS}(\vek{x})$ are
\begin{align*}
S_{2z-1}=z+\arcsin\left(-\frac{0.1}{2\pi}\right)\frac1{2\pi}\ \wedge\  
S_{2z}=z-\frac12-\arcsin\left(-\frac{0.1}{2\pi}\right)\frac1{2\pi}\ ,
\ z\in\mathbb{N}^+.
\end{align*}
The function is monotonic between those roots.
Thus, Theorem~\ref{th:sep} can be applied to the Salomon function on the
(sub-)domain
$\vek{x}\in\iv{\frac{S_z}{\sqrt{n}},\frac{S_{z+1}}{\sqrt{n}}}^n$ for all
$z\in\mathbb{N}^+$.
If $z$ is even, the minimum of the separator is required for a minimum of the
objective function.
Otherwise, if $z$ is odd the separator needs to be maximized to obtain a
minimum of the objective function.
\end{example}

Next, we show how to compute interval adjoints and how they can be
used to apply Theorem~\ref{th:sep} to a differentiable program implementing a
function $f$.
Furthermore, we use interval adjoints to verify structural separators.

\section{Implementation}
\label{sec:implementation}
Let $f:\mathbb{R}^n\to\mathbb{R}$ be implemented as a differentiable program
$y=f(\vek{x})$ with \emph{independent} variables $\vek{x}$ and \emph{dependent}
variable $y$.
Following \cite{Griewank2008}, we assume that at a particular argument
$\vek{x}$ the implementation of $f$ can be expressed by a finite sequence of
elemental function evaluations as
\begin{align}
\label{eq:SAC}
\begin{alignedat}{2}
v_i &= x_i\ ,\quad && i=0,\ldots, n-1\ ,\\
v_j &= \varphi_j(v_i)_{i\prec j}\ ,\quad && j=n,\ldots,n+p\ ,\\
y &= v_{n+p}\ ,
\end{alignedat}
\end{align}
where $v_j$ for $j=n,\ldots,n+p-1$ are referred to as \emph{intermediate}
variables.
The precedence relation $i\prec j$ indicates a direct dependency of $v_j$ on
$v_i$.
Furthermore, the transitive closure $\prec^\ast$ of $\prec$ induces a partial
ordering of all indices $j=0,\ldots,n+p$.
Equation \eqref{eq:SAC} is also referred to as the single assignment code~(SAC)
of $f$.
The SAC may not be unique due to commutativity, associativity and
distributivity.
We assume a SAC to be given.

\subsection{Interval Arithmetic}
\label{sec:IA}
Interval arithmetic~(IA) is a concept that enables the computation of bounds of a
function evaluation on a given interval.
A closed interval of a variable $x$ with lower bound $\lb{x}$ and upper bound
$\ub{x}$ is denoted as
\begin{align*}
\iv{x} = \iv{\lb{x},\ub{x}}
= \left\{x\in\mathbb{R}\ |\ \lb{x}\leq x\leq\ub{x}\right\}\ .
\end{align*}
If there is only a single element in $\iv{x}$, i.e, the endpoints are equal
$\lb{x}=\ub{x}$, then the square brackets $\iv{\cdot}$ are dropped and
$x$ is called a degenerate interval.
In that sense IA represents an extension of the real/floating-point number
system.

Interval vectors are denoted by bold letters and have endpoints for each
component
\begin{align*}
\iv{\vek{x}}= \iv{\lb{\vek{x}},\ub{\vek{x}}}
=\left\{\vek{x}\in\mathbb{R}^n\ |\ \lb{x_i}\leq x_i\leq\ub{x_i}\right\}\ .
\end{align*}
When evaluating a function $y=f(\vek{x})$ in IA on $\iv{\vek{x}}$ we are
interested in the information
\begin{align*}
\iv{y}=f^\ast(\iv{\vek{x}})=\{f(\vek{x})\ |\ \vek{x}\in\iv{\vek{x}}\}\ .
\end{align*}
The asterisk denotes the united extension which computes the true range of
values on the given domain.
United extensions for all unary and binary elementary functions and
arithmetic operations are known and endpoint formulas can be looked up e.g.\ in
\cite{Moore2009}.
Unfortunately, the derivation of endpoint formulas for the united extensions of
composed functions might be expensive or even impossible.
Hence, we will compute corresponding estimates by natural interval extensions.
A natural interval extension can be obtained by replacing all elemental
functions $\varphi_j$ in \eqref{eq:SAC} with their corresponding united
extensions as
\begin{align}
\label{eq:NIESAC}
\begin{alignedat}{2}
\iv{v_i}&=\iv{x_i}\ ,\quad && i=0,\ldots, n-1\ ,\\
\iv{v_j}&=\varphi^\ast_j(\iv{v_i})_{i\prec j}\ ,\quad&&j=n,\ldots,n+p\ ,\\
\iv{y}&=\iv{v_{n+p}}\ .
\end{alignedat}
\end{align}
The computation of the interval function value by the natural interval
extension from \eqref{eq:NIESAC} results in
\begin{align*}
\iv{y}=f(\iv{\vek{x}})\supseteq f^\ast(\iv{\vek{x}})\ .
\end{align*}
The superset relation states that the interval $\iv{y}$ can be an
overestimation of all possible values over the given domain, but it guarantees
enclosure.
Furthermore, the natural interval extension of Lipschitz continuous functions
converges linearly to the united extension with decreasing domain size.

The reader is referred to \cite{Moore1979,Moore2009,Hansen2004,Hansen1969}
for more information on the topic.

\subsection{Adjoint Algorithmic Differentiation}
\label{sec:AD}
Algorithmic differentiation~(AD) techniques \cite{Griewank2008,Naumann2012} use
the chain rule to compute in addition to the function value of a \emph{primal}
implementation its derivatives with respect to independent variables at a
specified point.

The adjoint or backward mode of AD propagates derivatives of the function
w.r.t.\ independent and intermediate variables in reverse relative to the order
of their computation in the primal SAC.
The computationally intractable combinatorial optimization problem known as
\textsc{DAG Reversal} \cite{Naumann2009} is implied.

Following \cite{Naumann2012}, first-order adjoints are marked with a
subscript $_{(1)}$.
They are defined as
\begin{align*}
\vek{x}_{(1)} =  y_{(1)} \cdot \frac{df}{d\vek{x}}(\vek{x})\ .
\end{align*}
A single adjoint computation with \emph{seed} $y_{(1)} = 1$ results in the
gradient $\frac{df}{d\vek{x}}(\vek{x})$ stored in $\vek{x}_{(1)}$.

The adjoint of \eqref{eq:SAC} can be implemented by \eqref{eq:SAC} itself
followed by
\begin{align}
\label{eq:ASAC}
\begin{alignedat}{2}
v_{(1),n+p} &= y_{(1)}\ ,\\
v_{(1),k} &= \sum_{j:k\prec j}
v_{(1),j} \cdot \frac{\partial\varphi_j}{\partial v_k}(v_i)_{i\prec j}
\ ,\quad && k=n+p,\ldots,n\ ,\\
x_{(1),i} &= v_{(1),i}\ ,\quad && i=n-1,\ldots, 0\ .
\end{alignedat}
\end{align}
The evaluation of the adjoint yields the adjoints of all intermediate variables
$v_j$
\begin{align*}
v_{(1),j} &= y_{(1)} \cdot \frac{df}{dv_j}(\vek{x})\ ,\quad j=n+p,\ldots,n\ .
\end{align*}
%

\subsection{Interval Adjoints}
\label{sec:IAAD}
The natural interval extension of \eqref{eq:SAC} and \eqref{eq:ASAC}
yields the interval function value and its interval derivatives w.r.t.\ all
independent and intermediate variables as the result of a single evaluation.
It can be implemented as \eqref{eq:NIESAC} followed by
\begin{align}
\label{eq:NIEASAC}
\begin{alignedat}{2}
\iv{v_{(1),n+p}} &= \iv{y_{(1)}}\ ,\\
\iv{v_{(1),k}} &= \sum_{j:k\prec j}
\iv{v_{(1),j}}\cdot
\frac{\partial\varphi^\ast_j}{\partial v_k}(\iv{v_i})_{i\prec j}
\ ,\quad && k=n+p,\ldots,n\ ,\\
\iv{x_{(1),i}} &= \iv{v_{(1),i}}\ ,\quad && i=n-1,\ldots, 0\ .
\end{alignedat}
\end{align}
Compared to the traditional approach of AD in which the derivatives are only
computed at specified points, we now get globalized derivatives that contain
all possible values of the derivative over the specified domain.
The interval adjoints in \eqref{eq:ASAC} might be overestimated compared
to the united extension as it is already stated for the interval values in
Section~\ref{sec:IA}.
The natural interval extension of the adjoint converges linearly for
continuously differentiable functions \cite{Deussen2021}.
Higher-order converging interval extensions of adjoints can be derived,
e.g.\ by centered forms.

\subsubsection{Monotonicity Check}
A single evaluation of the interval adjoint for $\iv{y_{(1)}}=1$ suffices to
verify monotonicity as in \eqref{eq:cond} for all independent and intermediate
variables.
If the separation approach is embedded into a b\&b solver that involves
verification of the first-order optimality condition by interval adjoints, then
the monotonicity check is for free, assuming that the separators are known
apriori.

\subsubsection{Verification of Separators}
Interval adjoints can be used to detect if an intermediate
variable $s$ is a separator.
Note that $\frac{df}{ds}(\iv{\vek{x}})$ as well as
$\frac{df}{dx_i}(\iv{\vek{x}})$ are assumed to be available from the adjoint
evaluation required for the monotonicity check.
An additional evaluation of \eqref{eq:NIEASAC} is required with the adjoint of
the intermediate variable set to $\iv{s_{(1)}}=\frac{df}{ds}(\iv{\vek{x}})$.
The resulting adjoints of the independent variables become equal to
\begin{align*}
\iv{x_{(1),i}}=\frac{ds}{dx_i}(\iv{\vek{x}})\cdot\frac{df}{ds}(\iv{\vek{x}})\ .
\end{align*}
If $f$ is structurally separable and fulfills Definition~\ref{def:gs} with
separator $s$, then
\begin{alignat*}{2}
\frac{df}{dx_i}(\vek{x})&=
\frac{ds}{dx_i}\big((x_j)_{j\in X_1}\big)\cdot
\frac{df}{ds}(\vek{x})
\quad &&\forall i \in X_1\ ,
\end{alignat*}
needs to hold over the entire domain, which can be verified by
\begin{align}
\iv{x_{(1),i}}=\frac{df}{dx_i}(\iv{\vek{x}})\quad\forall i \in X_1\ ,
\label{eq:sv1}
\end{align}
and since
\begin{align*}
\frac{ds}{dx_i}\big((x_j)_{j\in X_1}\big) = 0\quad\forall i \in X_2\ ,
\end{align*}
all other independent variables need to satisfy
\begin{align}
\iv{x_{(1),i}}=0 \quad\forall i \in X_2\ .
\label{eq:sv2}
\end{align}

If any $\iv{x_{(1),i}}$ fulfills neither \eqref{eq:sv1} nor \eqref{eq:sv2},
then $s$ is not a separator.
Consequently, in addition to the interval adjoint evaluation for the
monotonicity check another interval adjoint evaluation is required for the
verification of each separator candidate.

An exhaustive search for separators should be avoided, due to the potentially
high number of intermediate variables and the associated number of separator
candidates.
Separators given by expert users can be verified efficiently.
Since structural separability as given in Definition~\ref{def:gs} is
domain-independent and thus is a global property, it is sufficient to identify
the separators once before performing the global search.

\section{Case Study}
\label{sec:casestudy}
The general idea of b\&b algorithms~\cite{Hansen2004} used for
global optimization problems as given in \eqref{eq:opt} is to
remove all parts of the domain that cannot contain a global minimum.
The implementation used for this case study is a variation of
the one presented in \cite{Deussen2019} implementing Theorem~\ref{th:sep}.
The user needs to specify at least one separator.
The algorithm performs the following steps:
\begin{itemize}
\item \emph{bisection}:
half-splitting in every dimension resulting in $2^n$ subdomains;
\item \emph{value check}:
elimination of subdomain $\iv{\vek{x}}$ if $f(\iv{\vek{x}}) > \ub{y}^\ast$ with
upper bound $\ub{y}^\ast$ for the global minimum;
\item \emph{first-order optimality check}:
If $\frac{df}{dx_i}(\iv{\vek{x}})\geq0$ and $\lb{x_i}$ is a bound of original
domain $D$, then recompute with $x_i=\lb{x_i}$,
else if $\frac{df}{dx_i}(\iv{\vek{x}})\leq0$ and $\ub{x_i}$ is a bound of
original domain $D$, then recompute with $x_i=\ub{x_i}$, otherwise eliminate
subdomain $\iv{\vek{x}}$;
\item \emph{improvement of bound} $\ub{y}^\ast$:
Evaluate the function at any point (e.g.\ midpoint) of the subdomain to find a
better bound $\ub{y}^\ast$;
\item \emph{separator check}:
Check monotonicity condition for apriori known separators and generate a
subproblem if Theorem~\ref{th:sep} is applicable.
\end{itemize}
Obviously, the improvement of the upper bound of the global minimum can
be enhanced by local searches instead of evaluation of the objective function
at the midpoint of the current subdomain.
Recursive separation is not supported by the current version of the solver.
It is the subject of ongoing development efforts.

The software implements the required interval adjoints by using the
interval type from the \texttt{Boost} library \cite{boost} as a base type of
the first-order adjoint type provided by
\dcocpp\footnote{\url{https://www.nag.co.uk/content/adjoint-algorithmic-differentiation}}
\cite{dco}.
Both template libraries make use of the concept of operator overloading as
supported e.g.\ by C++.

\begin{figure}[!tb]
\centering
\subfloat{\includegraphics[width=0.5\textwidth]{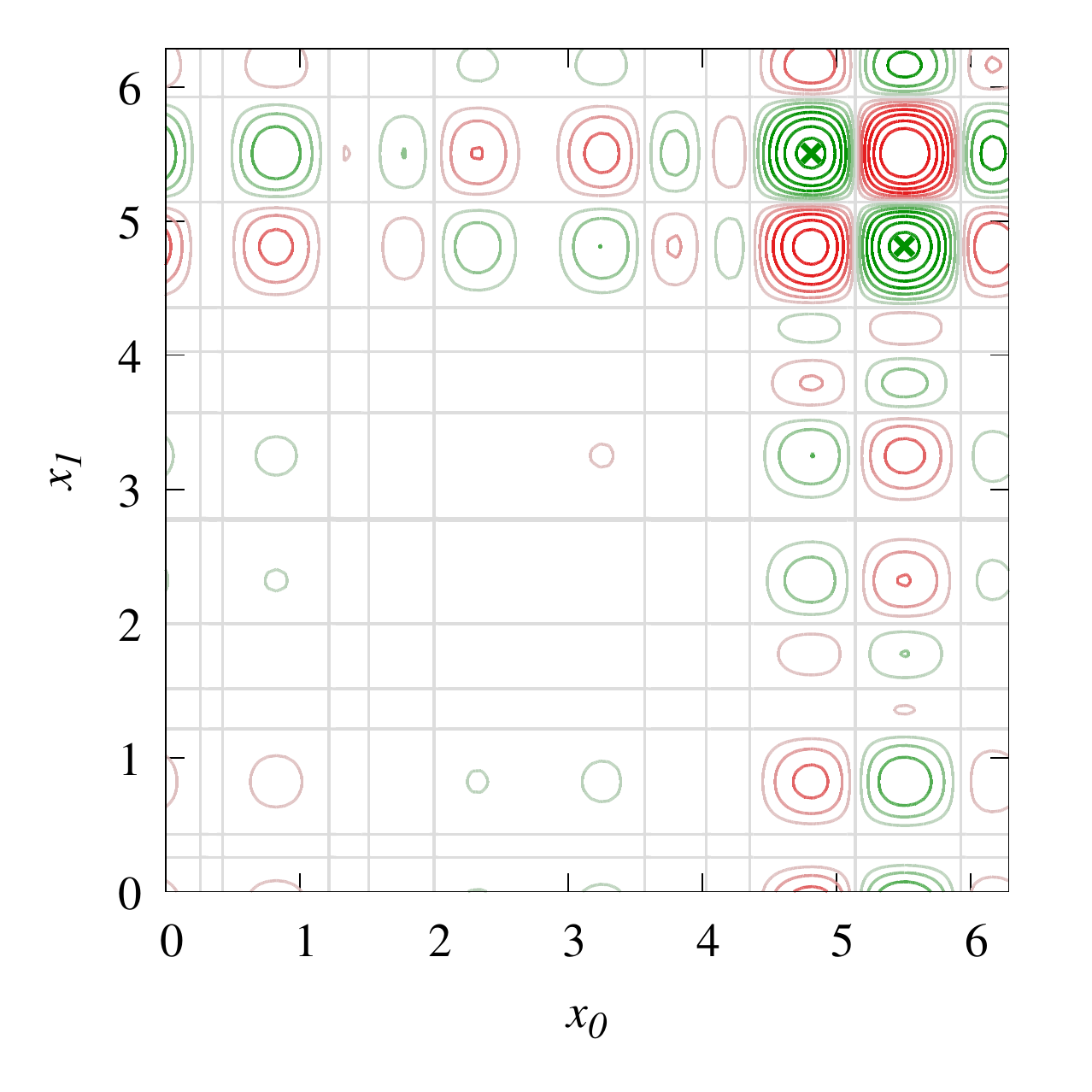}}
\subfloat{\includegraphics[width=0.5\textwidth]{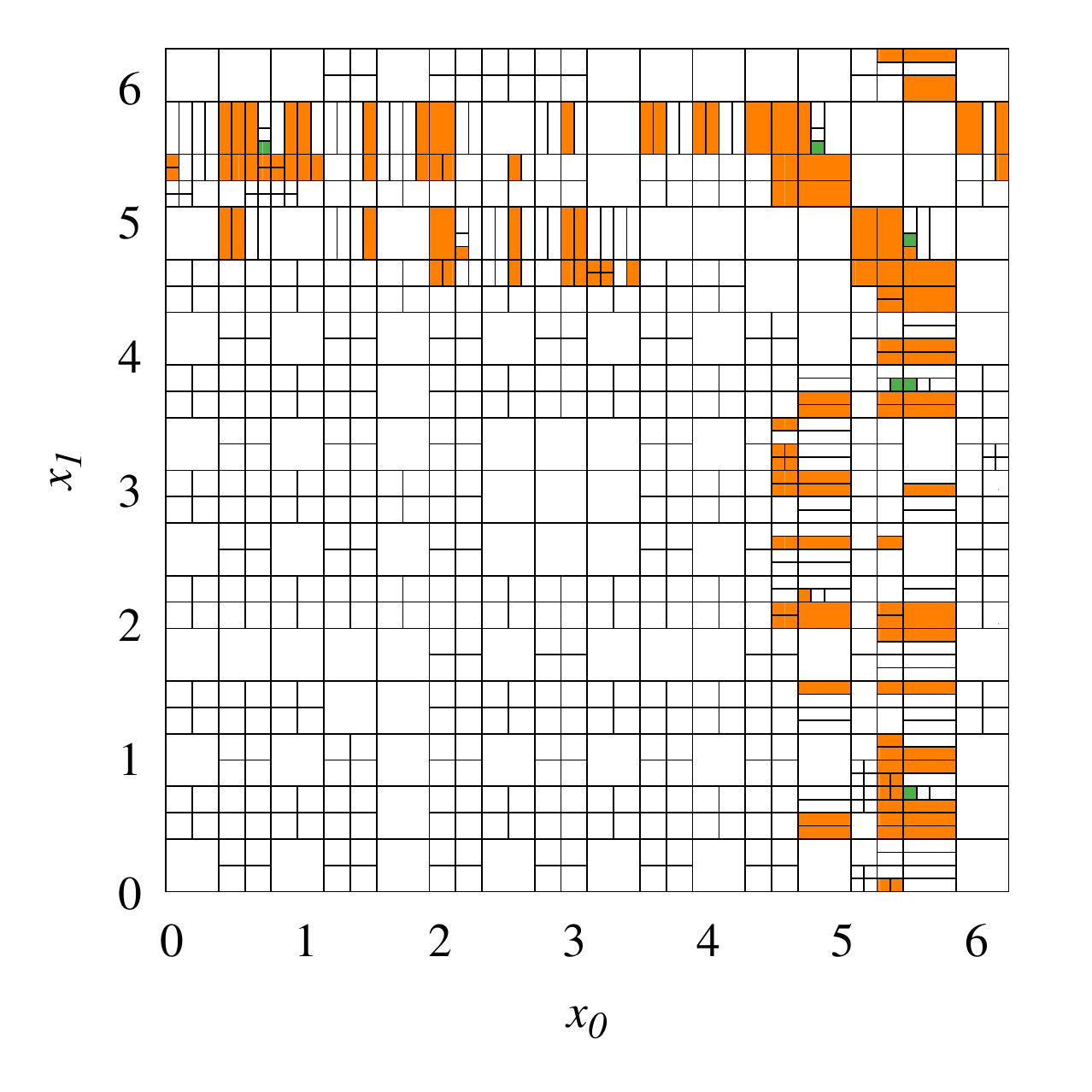}}
\caption{Isolines of the Shubert function for $n=2$ (left) with green lines
around local minima and red lines around local maxima. Subdomains considered
by the b\&b algorithm (right) with active subdomains marked in green,
white subdomains are discarded by the value check and by the first-order
optimality condition is violated on orange subdomains.
Non-square subdomains result from the separation approach.}
\label{fig:sep}
\end{figure}
\begin{table}[!tb]
\caption{Number of generated (sub-)domains by the b\&b algorithm
without and with separation.}
\centering
\small
  \begin{tabularx}{0.98\linewidth}{Xj{0.1\textwidth}p{0.2\textwidth}|
  p{0.05\textwidth}p{0.15\textwidth}p{0.15\textwidth}}
  \toprule
   & $n$ & domain && w/o sep. & w/ sep. \\
  \midrule 
  Styblinski-Tang & 4 & $\iv{-5,5}^4$ && 4609 & 285 \\
                  & 8 & $\iv{-5,5}^8$ && 5018817 & 569 \\
  Exponential & 4 & $\iv{-1,1}^4$ && 18 & 17 \\
              & 8 & $\iv{-1,1}^8$ && 258 & 33 \\
  Recursive Exponential & 4 & $\iv{-2.1,2.0}^4$ && 273 & 252 \\
                        & 8 & $\iv{-2.1,2.0}^8$ && 4609 & 549 \\
  Shubert & 4 & $\iv{-10,10}^4$ && 248618257 & 5272861 \\
  Salomon & 4 & $\iv{-100,100}^4$ && 2322 & 2322 \\
          & 8 & $\iv{-100,100}^8$ && 655618 & 655618 \\
  \bottomrule
  \end{tabularx}
\label{tab:bnbnodes}
\end{table}
On the left side of Fig.~\ref{fig:sep} isolines of the two-dimensional Shubert
function over the domain $\iv{0,2\pi}$ are shown with green lines around
(local) minima and red lines around local maxima.
The two global minima are marked by green crosses.
The right side of Fig.~\ref{fig:sep} shows the subdomains that are considered
by the b\&b algorithm.
For visualization the branching is set up to stop when the subdomain is smaller
than $0.1$ in any direction.
Non-square domains result from the separation approach and only occur in
regions that are proven to be monotonic by the interval adjoints.
Green boxes are \emph{active} domains that could contain the global minimum.
White boxes are discarded by the value check.
Orange boxes violate the first-order optimality condition.

Our solver is used to find the global minima of the examples from
Section~\ref{sec:contribution}.
The algorithm is performed with and without separation.
Structural separators are marked manually.
The results are summarized in Table~\ref{tab:bnbnodes}.
Most of the presented examples benefit from the domain-dependent separation
approach and have less subdomains generated by b\&b if separation is
enabled.
The benefit increases with growing dimensionality due to the exponential
complexity of the bisection.
The Salomon function does not benefit from the domain-dependent separation
since the relevant domains are already discarded by the value or first-order
optimality checks. 

We only measure runtimes for the Styblinski-Tang example with $n=8$ with and
without exploiting subdomain separability.
Since the derivative information is already available for all separators after
the first-order optimality check, the monotonicity check only iterates over the
separators defined by the user.
The number of subdomains considered by the b\&b algorithm
without separation is $8820$ times higher than with separation.
The corresponding runtime without separation is only 7673 times higher than
with separation.
This observation correlates with the fact that the computations of subdomains
that do not pass the value check are terminated immediately.
The percentage of subdomains that are eliminated due to the value check is
$30.2\%$ for the case without separation and $2.8\%$ with separation
approach.
The runtime estimates are averaged over 100 calls of the solver for both cases.

Our in-house solver has been designed as a playground for novel algorithms.
Neither is it optimized for speed, nor does it feature state-of-the-art
non-convex optimization methodology beyond the previously described b\&b
algorithm.
Ultimately, we aim for integration of our ideas into modern software solutions
for global optimization, e.g.\ \cite{baron,maingo}.


\section{Conclusion and Outlook}
\label{sec:discussion}
\label{sec:conclusion}
Our notion of separability combined with checks for monotonicity allows us to
decompose an optimization problem into smaller optimization problems.
It extends the verification of the first-order optimality condition as it was
proposed in \cite{Salomon1996}.
This also enables implementation of the proposed work as an add-on to
deterministic global optimization algorithms by considering all
possible optima instead of some candidates fulfilling first-order optimality
condition.
We explained how to utilize interval adjoints to verify monotonicity of
the objective function w.r.t.\ all structural separators at the cost of a
single adjoint evaluation.
As a first result, we revisited examples from the literature that
benefit from the domain-dependent separability approach.
Furthermore, we showed how to verify the separation property of a variable
in a given computer program at the cost of only two adjoint evaluations.

The verification of separators can be used as a starting point for research
into heuristics for automatically detecting separators in a computer program.
Further work in progress includes enabling recursive separation.
Moreover, interval arithmetic can result in a significant overestimation of the
true value range, e.g.\ due to the wrapping effect or the dependency problem.
The replacement of interval adjoints by an adjoint version of affine arithmetic
\cite{Messine2002} or by McCormick relaxations
\cite{McCormick1976,Mitsos2009,Deussen2020} of adjoints is expected to yield
tighter enclosures.


\end{document}